\documentclass[a4paper,12pt]{amsart}%,final
\usepackage{graphics}
\usepackage{graphicx}
\usepackage{amssymb}
\usepackage{amsmath}
\usepackage{amsthm}

\usepackage{todonotes}

\usepackage{multicol}

\usepackage[english]{babel}
\usepackage{cmap}

\usepackage{setspace}

\usepackage{comment}

\usepackage{enumerate}

\usepackage{setspace}
\makeatletter
 
  \@addtoreset{equation}{section}
\makeatother
\makeatletter
\def\tbcaption{\def\@captype{table}\caption}
\makeatother
\newtheorem{theorem}{Theorem}[section]
\newtheorem{lemma}[theorem]{Lemma}
\newtheorem{proposition}[theorem]{Proposition}
\newtheorem{cor}[theorem]{Corollary}
\theoremstyle{definition}
\newtheorem{definition}[theorem]{Definition}

\newtheorem{conjecture}{Conjecture}[section]

\theoremstyle{remark}
\newtheorem{remark}[theorem]{Remark}

\numberwithin{equation}{section}

\newcommand{\beq}[1]{\begin{equation}\label{#1}}
\newcommand{\eeq}{\end{equation}}

\def\KK{\mathbb K}
\def\QQ{\mathbb Q}
\def\ZZ{\mathbb Z}
\def\RR{\mathbb R}
\def\NN{\mathbb N}

\begin{document}
%\doublespacing
%%%%%%%%%%%%%%%%%%%%%%%%%%%%%%%%%%%%%%%%%%%%%%%%%%%%%%%%%%%
\title[Extensions of $D(4)$-pairs $\{a, ka\}$ with $k\in \{7,8,10,11,12,13\}$]{ Extensions of $D(4)$-pairs $\{a, ka\}$ with $k\in \{7,8,10,11,12,13\}$}

\author[M. Bliznac Trebje\v{s}anin and P. Radić]{Marija {Bliznac Trebje\v{s}anin} and Pavao Radić}
\date{\today}

\begin{abstract}
We study the extensibility of $D(4)$-pairs $\{a,b\}$, where $b = ka$ and $k \in \{7,8,10,11,12,13\}$. Firstly, we show that it can be extended to a $D(4)$-triple with an element c, which is a member of a family of positive integers depending on a. Then, we prove that such a triple has a unique extension to a $D(4)$-quadruple.
\end{abstract}
\maketitle


%%%%%%%%%%%%%%%%%%%%%%%%%%%%%%%%%%%%%%%%%%%%%%%%%%%%%%%%%%%
\section{Introduction}\label{intr}
%%%%%%%%%%%%%%%%%%%%%%%%%%%%%%%%%%%%%%%%%%%%%%%%%%%%%%%%%%%

\begin{definition}Let $n\neq0$ be an integer. We call a set of $m$ distinct positive integers a $D(n)$-$m$-tuple, or $m$-tuple with the property $D(n)$, if the product of any two of its distinct elements increased by $n$ is a perfect square.
\end{definition}
We research the $n=4$ case, which has many similarities to the classical $n=1$ case. First author and Filipin proved in \cite{blizfil} the nonexistence of $D(4)$-quintuples.

For a $D(4)$-triple $\{a,b,c\}, a<b<c$, we define
 \[ d_{\pm}(a,b,c)=a+b+c+\frac{1}{2}(abc \pm \sqrt{(ab+4)(ac+4)(bc+4)}).\]
 
$D(4)$-quadruple  $\{a,b,c,d_{+}\}$ is called a \textbf{regular quadruple}
 if $d_{-} \neq 0$, then $\{a,b,c,d_{-}\}$ is a regular $D(4)$-quadruple with $d_{-}<c$. It is easy to verify that $c=d_+(a,b,d_-)$.

In both the classical case, $n=1$, and $n=4$, conjectures about the uniqueness of an extension of a triple to a quadruple with a larger element are still open.

\begin{conjecture}\label{conj}
Any $D(4)$-quadruple is regular.
\end{conjecture}

In this paper, we study the extensibility of $D(4)$-pairs $\{a,b\}$, where $b = ka$ and $k \in \{7,8,10,11,12,13\}$ are defined by following the methods in \cite{glavni}. So, the main theorem we prove is:
\begin{theorem}\label{main result}
Let $k$ be a positive integer such that $k\in \{7, 8, 10,$ $11,12,13\}.$ If $\{a, b, c, d\}$ is a $D(4)$-quadruple with $b=ka$, then it is regular. In other words, we have $d=d_{\pm}.$
\end{theorem}
The cases $k \in \{2,3,5,6\}$ have already been observed in \cite{glavni} and \cite{Filipin}, and it is easy to verify that pairs $\{a,4a\}$ and $\{a,9a\}$ cannot satisfy the $D(4)$ property.

Firstly, in Section $2$ we show that $D(4)$-pairs $\{a,ka\}$, $k\in \{7, 8, 10,$ $11,12,13\}$ can be extended to a $D(4)$-triple only with an element $c$, which is a member of a family of positive integers depending on $a$. Then in Sections $3-5$, we prove that these triples can only be extended to regular $D(4)$-quadruples $\{a, ka, c, d\}$.\ These examples support the Conjecture \ref{conj}.
%%%%%%%%%%%%%%%%%%%%%%%%%%%%%%%%%%%%%%%%%%%%%%%%%%%%%%%%%%%
\section{Extensions of pairs to triples}\label{section_2}
%%%%%%%%%%%%%%%%%%%%%%%%%%%%%%%%%%%%%%%%%%%%%%%%%%%%%%%%%%%

If $\{a,ka\}$ is a $D(4)$-pair, then there exists $r\in \NN$ such that
\begin{equation}\label{pair eqn}
ka^2+4=r^2.
\end{equation}
Rewriting (\ref{pair eqn}) as a Pellian equation, yields 
\begin{equation}\label{pair Pellian} r^2-ka^2=4.
\end{equation}
It is easy to verify that there is only one fundamental solution $(r_1,a_1)$ of (\ref{pair Pellian}), for any $k=7,8,10,11,12,13$, namely $(r_1,a_1)\in\{(16,6),(6,2),$
$(38,12),(20,6),$
$(4,1),(11,3)\}$, respectively. All solutions $(r_p,a_p)$ of the equation (\ref{pair Pellian}) are given by
\begin{align}\label{sequence-a_p}
\frac{r_p+a_p\sqrt{k}}{2}=\left(\frac{r_1+a_1\sqrt{k}}{2}\right)^p,\quad p\in\mathbb{N}.
\end{align}
From this relation, if we denote the $p$-th element of the sequence for each $k$ with $a_p^{(k)}$, we easily obtain the following relations:
\begin{align}
    a_p^{(7)}&=\dfrac{1}{\sqrt{7}}\left((8+3\sqrt{7})^p-(8-3\sqrt{7})^p \right),\label{a_p_1}\\
    a_p^{(8)}&=\dfrac{1}{\sqrt{8}}\left((3+\sqrt{8})^p-(3-\sqrt{8})^p \right),\label{a_p_2}\\
    a_p^{(10)}&=\dfrac{1}{\sqrt{10}}\left((19+6\sqrt{10})^p-(19-6\sqrt{10})^p \right),\label{a_p_3} 
    \end{align}
    \begin{align}
    a_p^{(11)}&=\dfrac{1}{\sqrt{11}}\left((10+3\sqrt{11})^p-(10-3\sqrt{11})^p \right),\label{a_p_4}\\
    a_p^{(12)}&=\dfrac{1}{\sqrt{12}}\left(\left(\frac{4+\sqrt{12}}{2}\right)^p-\left(\frac{4-\sqrt{12}}{2}\right)^p \right),\label{a_p_5}\\
    a_p^{(13)}&=\dfrac{1}{\sqrt{13}}\left(\left(\frac{11+3\sqrt{13}}{2}\right)^p-\left(\frac{11-3\sqrt{13}}{2}\right)^p \right).\label{a_p_6}
\end{align}

As we will present, many proofs of the lemmas in this article require the following fact, which is easy to verify. 
$$
\gcd(a_p,r_p)=  \begin{cases}
2,&\ k=7,8,10,11,\\
2,&\ k=12, p \equiv 0\ (\bmod 2),\\
1,&\ k=12, p \equiv 1\ (\bmod 2),\\
2,&\ k=13, p \equiv 0\ (\bmod 3),\\
1,&\ k=13, p \equiv 1,2\ (\bmod 3).

\end{cases}
$$
If we assume that an irregular $D(4)$-quadruple exists, from \cite{mbt} we know a numerical lower bound on element $b$:
\begin{lemma}\cite[Lemma 2.2]{mbt}\label{lower_on_b}
Let $\{a,b,c,d\}$ be a $D(4)$-quadruple such that $a<b<c<d_+<d$. Then $b>10^5$.
\end{lemma}
It is straightforward to use this bound to determine the lower bound for $a$ in an irregular $D(4)$-quadruple $\{a,ka,c,d \}$. This is presented in Table~\ref{tab:lowerbounds}.

\begin{table}[ht]
\caption{Lower bounds for the element $a$}
\label{tab:lowerbounds}       % Give a unique label
\begin{tabular}{|c|c|c|c|c|c|c|}
\hline
        $k$ & $7$ & $8$ & $10$ & $11$ & $12$ & $13$ \\
        \hline
        $p$ & $4$ & $6$ & $3$ & $4$ & $8$ & $5$ \\
        \hline
        $a_p^{(k)}$ & $24384$ & $13860$ & $17316$ & $47760$ & $10864$ & $42837$ \\
        \hline
\end{tabular}
\end{table}
 Let $\{a,b\}$ be a $D(4)$-pair. Then there exists a positive integer $r$ such that $ab+4=r^2$. Extending this pair to a $D(4)$-triple with an element $c$ means finding $s,t \in \mathbb{N}$ such that $$
    ac+4=s^2,\quad bc+4=t^2.$$
These two equations yield a Pellian equation
\begin{equation}\label{eqn: pell for c}
    at^2-bs^2=4(a-b).
\end{equation}
Its solutions $(t,s)$ are given by
\begin{equation}
 (t_{\nu}\sqrt{a} + s_{\nu}\sqrt{b})
 = (t_0\sqrt{a} + s_0\sqrt{b})
 \left(\frac{r + \sqrt{ab}}{2}\right)^{\nu},
 \quad \nu \geq 0,
\label{eq:ts_solutions}
\end{equation}
where $(t_0,s_0)$ is a fundamental solution of the equation (\ref{eqn: pell for c}) and $\nu$ is a nonnegative integer.

It has been proved in \cite[Lemma 6.1]{mbt} that for $b\leq 6.85a$ the only fundamental solutions of this equation from which we obtain the third element $c$, which is an integer, are $(t_0,s_0)=(\pm2,2)$. 
From \cite{Filipin-LjB drugi} we know that all solutions $(t_{\nu},s_{\nu})$ generated with $(t_0,s_0)=(\pm2,2)$ can be represented by recursively defined sequences:
\begin{align}
    &t_0=\pm 2,\ t_1=b\pm r,\ t_{\nu+2}=rt_{\nu +1}-t_{\nu},\label{sequence_t_nu}\\
    &s_0= 2,\ s_1=r\pm a,\ s_{\nu+2}=rs_{\nu +1}-s_{\nu},\ \nu\geq 0.\label{sequence_s_nu}
\end{align}
Since $c=\frac{s^2-4}{a}$, an explicit expression for the third element $c$ in terms of $a$ and $b$ is given by
\begin{align}
c=c_{\nu}^{\pm}=\frac{4}{ab}&\left\{\left(\frac{\sqrt{b}\pm\sqrt{a}}{2}\right)^2\left(\frac{r+\sqrt{ab}}{2}\right)^{2\nu}\right. \label{eqn:c_nu_general_exp}\\
&\left.
+\left(\frac{\sqrt{b}\mp\sqrt{a}}{2}\right)^2\left(\frac{r-\sqrt{ab}}{2}\right)^{2\nu}-\frac{a+b}{2}\right\},\nonumber
\end{align}
where $\nu\geq 0$ is an integer. The first few elements of this sequence are:
\begin{align*}
c_1^{\pm}&=a+b\pm 2r,\\
  c_2^{\pm}&=(ab+4)(a+b\pm2r)\mp4r,\\
  c_3^{\pm}&=(a^2b^2+6ab+9)(a+b\pm2r)\mp4r(ab+3), \\
  c_4^{\pm}&=(a^3b^3+8a^2b^2+20ab+16)(a+b\pm2r)\mp4r(a^2b^2+5ab+6).
\end{align*}

\begin{lemma}\label{pomocna lema za lemu ispod}
   Let $\{a,b\}$ be a $D(4)$-pair. Then $d_+(a,b,c_{\nu}^{\pm})=c_{\nu+1}^{\pm}$ and $d_-(a,b,c_{\nu}^{\pm})=c_{\nu-1}^{\pm}$.
\end{lemma}
\begin{proof}
    From (\ref{eq:ts_solutions}) we get $s_{\nu+1}=\frac{r}{2}s_{\nu}+\frac{a}{2}t_{\nu}$. Since $c_{\nu}^{\pm}=\frac{(s_{\nu})^2-4}{a}=\frac{(t_{\nu})^2-4}{b}$ it follows that $c_{\nu+1}^{\pm}=\frac{(s_{\nu+1})^2-4}{a}=a+b+c_{\nu}^{\pm}+\frac{1}{2}(abc_{\nu}^{\pm}+rs_{\nu}t_{\nu})=d_+(a,b,c_\nu^{\pm})$.
\end{proof}
The main goal of this section is to improve this upper bound for $b$ in terms of $a$ under an additional restriction.
\begin{lemma}\label{glavna nova lema}
    Let $\{a,b,c\}$ be a $D(4)$-triple and $a< b\leq 13.92a$. Suppose that $\{1,5,a,b\}$ is not a $D(4)$-quadruple. Then $c=c_{\nu}^{\pm}$ for some positive integer $\nu$.
\end{lemma}
\begin{proof}
    We follow the idea of \cite[Lemma 6.1]{mbt} and \cite[Lemma 1]{Filipin-LjB drugi}.
    Define $s'=\frac{rs-at}{2}, t'=\frac{rt-bs}{2}$ and $c'=\frac{{(s')}^2-4}{a}$. The cases $c'>b, c'=b$ and $c'=0$ are the same as in \cite[Lemma 1]{Filipin-LjB drugi} and yield $c=c_{\nu}^{\pm}.$ It is only left to consider the case $0<c'<b$. Here we define $r'=\frac{s'r-at'}{2}$ and $b'=\frac{(r')^2-4}{a}$. If $b'=0$ then it can be shown that $c'=c_1^-$ and $c=c_{\nu}^-$ for some positive integer $\nu$. Notice that $b'=d_-(a,b,c')$, hence, $$ b'<\frac{b}{ac'}\leq\frac{13.92a}{ac'}=\frac{13.92}{c'} \implies b'c'<13.92.$$
    Since $b'>0$ and $b'c'+4$ is a square, we consider the following cases for $b'$, depending on the term $c'$:
\begin{table}[h!]

    \centering
    \begin{tabular}{|c|c|c|c|c|c|c|c|c|}
    \hline
         $c'$&  $1$&$2$&$3$&$4$&$5$&$6$&$12$\\
         \hline
         $b'$& $5,12$& $6$& $4$& $3$& $1$& $2$ &$1$\\
         \hline
    \end{tabular}
    
\end{table}
\\Cases $c' \in \{7,8,9,10,11 \}$ and $c'\geq13$ imply there are no $b'>0$ that satisfy our conditions.
So, we obtain that $a$ and $b$ extend pairs $\{1,5\}, \{3,4\}, \{2,6\} ,\{1,12\}$. As our assumption is that $\{1,5,a,b\}$ is not a $D(4)$-quadruple, we only need to show that the three remaining cases also cannot be $D(4)$-quadruples when $b\leq 13.92a$. Let's suppose that $\{3,4,a,b\}$ is a $D(4)$-quadruple. From \cite[Lemma 6.1]{mbt}, it follows that 
\footnotesize
$$ a=a_{\nu}^{\pm}
=\frac{1}{3}\left\{\left(\frac{2\pm\sqrt{3}}{2}\right)^2\left(\frac{4+\sqrt{12}}{2}\right)^{2\nu} 
+ \left(\frac{\sqrt{2}\mp\sqrt{3}}{2}\right)^2\left(\frac{4-\sqrt{12}}{2}\right)^{2\nu}-\frac{7}{2} \right\},
$$
\normalsize
and from Lemma \ref{lower_on_b} it follows that $b=d_+(3,4,a)$. Lemma \ref{pomocna lema za lemu ispod} implies $d_+(3,4,a)=a_{\nu+1}^{\pm}$ for the same choice of $\pm$. Define $k := \frac{b}{a}=\frac{a_{\nu+1}^{\pm}}{a_{\nu}^{\pm}}$. It is easy to see that $k \leq 15.24$ and that it is decreasing as $\nu$ increases, and $$\lim_{\nu \rightarrow\infty} \frac{a_{\nu+1}^{\pm}}{a_{\nu}^{\pm}} = \left( \frac{4+\sqrt{12}}{2} \right)^2 > 13.92,$$, which gives us a contradiction to the assumption that $b \leq 13.92a$.
We use the same approach with two other cases and arrive at the same conclusion, the only difference being in case $\{1,12,a,b  \}$ because here we cannot use \cite[Lemma 6.1]{mbt} to show that $a=a_{\nu}^{\pm}$. Instead, we use bounds on the fundamental solutions of the corresponding Pellian equation from \cite[Theorem 10.21]{knjiga} .
\end{proof}

\begin{cor}\label{nova lema}
Let $\{a,ka,c\}$ be a $D(4)$-triple, $k \in \{10,11,12,13\}$. Then $c=c_{\nu}^{\pm}$.
\end{cor}
\begin{proof}
Let's show that $\{1,5,a,ka \}$ is not a $D(4)$-quadruple for $k \in \{10,11,12,13\}$. If $\{1,5,a,ka \}$ is a $D(4)$-quadruple, by using \cite[Lemma 6.1]{mbt}, Lemma \ref{lower_on_b}, and Lemma \ref{pomocna lema za lemu ispod}, we obtain $a=a_{\nu}^{\pm}$, $ka=d_+(1,5,a)$, and $d_+(1,5,a)=a_{\nu+1}^{\pm}$ for the same choice of $\pm$. We divide both sides of \begin{equation}\label{pomocna}
    ka = d_+(1,5,a)=1+5+a+\frac{1}{2}(5a+3\sqrt{5a+4}\sqrt{a+4})
\end{equation} by $a$ and, using the fact that $a \geq 12$, we obtain $k \leq 8$.
\end{proof}

\begin{lemma} \hfill
    \begin{enumerate}[i)]
    \item If $\{1,5,a,7a\}$ is a $D(4)$-quadruple, then $a=96$.
    \item If $\{1,5,a,8a\}$ is a $D(4)$-quadruple, then $a=12$.
\end{enumerate}
\end{lemma}
\begin{proof}
    From (\ref{pomocna}), when $k=7$ we get $a=96$ and when $k=8$ we get $a=12$.
\end{proof}
Using the theory of Pellian equations, it is easy to see that the only fundamental solutions for the corresponding Pellian equation when extending $\{96,672\}$ are $(\pm2,2)$ and $(\pm26,10)$, and when extending $\{12,96\}$ are $(\pm2,2)$ and $(\pm10,4)$. So we have two pairs of sequences for each of these pairs that extend them to  triples. Since the second element in both of these pairs is less than $10^5$, by Lemma \ref{lower_on_b} we know that all triples from those sequences extend only to regular quadruples.

Lemma \ref{glavna nova lema} allows us to further investigate the regularity of $D(4)$-quadrup\-les $\{a,ka,c,d\}$, i.e., we enhance \cite[Theorem 1.4]{glavni}.
In the following Sections, we show that $D(4)$-triples $\{a,ka,c \},\ k\in\{7,8,10,$
$11,12,13\}$, extend only to regular $D(4)$-quadruples.
Since we separately observed the exceptions, it remains to observe extensions of triples of the form $\{a,ka,c\}$, where $c=c_{\nu}^{\pm}$.
It is easy to see that $c_4^->a^3b^3$ and since $ka>10^5, \ k\in \{7,8,10,11,12,13\}$, we get from \cite[Theorem 1.6]{mbt} that  $c \in \{c_1^{\pm},c_2^{\pm},c_3^{\pm}\}$. Also, it is easy to see that inequalities $a<c_1^-<b$ and $c_1^+,c_2^{\pm},c_3^{\pm}>b$ hold in all our cases.

%%%%%%%%%%%%%%%%%%%%%%%%%%%%%%%%%%%%%%%%%%%%%%%%%%%%%%%%%%%
\section{Extensions of triples and linear forms in three logarithms}\label{section_3}

In this section, we observe a system of Pellian equations which corresponds to the extension of a $D(4)$-triple to a $D(4)$-quadruple. Then, we search for the intersection of linear recurrent sequences that describe solutions to these equations. To help us find these intersections, at the end of this section we use the theory of linear forms in logarithms to obtain some useful lemmas and results. Proofs that differ only in calculations from \cite{glavni} and \cite{ahpt} will be omitted.

\subsection{System of simultaneous Pellian equations}\label{subsec-3.1}

Let us observe an extension of a $D(4)$-triple $\{a,b,c\}$ to a $D(4)$-quadruple $\{a,b,c,d\}$. We need to find $x,y,z \in \mathbb{N}$ such that
\[
    ad+4=x^2,\quad
    bd+4=y^2,\quad
    cd+4=z^2.
\]
By eliminating $d$ from these equations, we obtain a system of generalized Pellian equations
\begin{align}
az^2-cx^2&=4(a-c),\label{eqn:pell_ac}\\
bz^2-cy^2&=4(b-c),\label{eqn:pell_bc}\\
ay^2-bx^2&=4(a-b).\label{eqn:pell_ab}
\end{align}
Its solutions $(z,x)$, $(z,y)$, and $(y,x)$ satisfy
\begin{align}
    z\sqrt{a}+x\sqrt{c}&=(z_0\sqrt{a}+x_0\sqrt{c})\left(\frac{s+\sqrt{ac}}{2}\right)^m,\label{eqn:sol_pell_ac}\\
    z\sqrt{b}+y\sqrt{c}&=(z_1\sqrt{a}+y_1\sqrt{c})\left(\frac{t+\sqrt{bc}}{2}\right)^n,\label{eqn:sol_pell_bc}\\
    y\sqrt{a}+x\sqrt{b}&=(y_2\sqrt{a}+x_2\sqrt{b})\left(\frac{r+\sqrt{ab}}{2}\right)^l,\label{eqn:sol_pell_ab}
\end{align}
where $m,n,l$ are nonnegative integers and $(z_0,x_0)$, $(z_1,y_1)$, and $(y_2,x_2)$ are fundamental solutions of (\ref{eqn:pell_ac})--(\ref{eqn:pell_ab}).

Firstly, we observe the solutions of the system of equations (\ref{eqn:sol_pell_ac}) and (\ref{eqn:sol_pell_bc}) and determine the intersections $z=v_m=w_n$ of sequences $(v_m)_m$ and $(w_n)_n$ defined by
\begin{align*}
&v_0=z_0,\ v_1=\frac{1}{2}\left(sz_0+cx_0\right),\ v_{m+2}=sv_{m+1}-v_{m},\\
&w_0=z_1,\ w_1=\frac{1}{2}\left(tz_1+cy_1 \right),\ w_{n+2}=tw_{n+1}-w_n.
\end{align*}
The initial terms of these sequences are described in the following theorem.
\begin{theorem}\cite[Theorem 1.3]{mbt}\label{tm:inital_terms}
Suppose that $\{a,b,c,d\}$ is a $D(4)$-quadruple with $a<b<c<d$ and that  $w_m$ and  $v_n$  are defined as before. 
\begin{enumerate}
\item[i)] If  the equation $v_{2m}=w_{2n}$ has a solution, then $z_0=z_1$ and $|z_0|=2$ or $|z_0|=\frac{1}{2}(cr-st)$.
\item[ii)] If the equation $v_{2m+1}=w_{2n}$ has a solution, then $|z_0|=t$, $|z_1|=\frac{1}{2}(cr-st)$ and $z_0z_1<0$.
\item[iii)] If the equation $v_{2m}=w_{2n+1}$ has a solution, then $|z_1|=s$, $|z_0|=\frac{1}{2}(cr-st)$ and $z_0z_1<0$.
\item[iv)] If the equation $v_{2m+1}=w_{2n+1}$ has a solution, then $|z_0|=t$, $|z_1|=s$ and $z_0z_1>0$.
\end{enumerate}
Moreover, if $d>d_+$, case $ii)$ cannot occur.
\end{theorem}

Under the assumption that some special $D(4)$-quadruples do not exist, we have the following lemma, which further reduces the number of cases for fundamental solutions we need to examine.

\begin{lemma}\cite[Lemma 2.2]{glavni}\label{lem:inital_terms}
Assume that $\{a,b,c,c'\}$ is not a $D(4)$-quadruple for any $c'$ with $0<c'<c_{\nu-1}^{\pm}$. We have 
\begin{enumerate}[i)]
    \item If the equation $v_{2m}=w_{2n}$ has a solution, then $z_0=z_1=\pm 2$ and $x_0=y_1=2$.
    \item If the equation $v_{2m+1}=w_{2n+1}$ has a solution, then $z_0=\pm t$, $z_1=\pm s$, $x_0=y_1=r$ and $z_0z_1>0$.
\end{enumerate}
\end{lemma}
\begin{comment}
\begin{proof}
The proof of this lemma is based on Theorem \ref{tm:inital_terms} and is similar to the proofs \cite[\mbox{Lemma}~3]{Filipin-LjB}, \cite[\mbox{Lemma}~2.3]{aft}, and \cite[\mbox{Lemma}~5]{Fujita}.
\end{proof}
    
\end{comment}

\begin{remark}\label{c_1-even}
If $c=c_1^{\pm}=a+b\pm 2r,$ then it is enough to observe the case $v_{2m}=w_{2n}$.
\end{remark}

Secondly, we observe the solutions of the system of equations (\ref{eqn:sol_pell_bc}) and (\ref{eqn:sol_pell_ab}) and determine the intersections $y=A_n=B_l$ of sequences $(A_n)_n$ and $(B_l)_l$ defined by
\begin{align}
    &A_0=y_1,\ A_1=\frac{1}{2}(ty_1+bz_1),\ A_{n+2}=tA_{n+1}-A_n,\label{seq:An}\\
    &B_0=y_2,\ B_1=\frac{1}{2}(ry_2+bx_2),\ B_{l+2}=rB_{l+1}-B_l, \ n,l\geq 0.\label{seq:Bl}
\end{align}
The initial terms of these sequences are described in the next lemma, whose proof follows \cite[Lemma 2.5]{glavni}.
\begin{lemma}\label{lem:lem_4}
Assume that $\{a, b, c', c\}$ is not a $D(4)$-quadruple for any $c'$ with $0<c'<c^{\pm}_{\nu-1}$ and $b\geq 556881$. Then, $A_{2n}=B_{2l+1}$ has no solution. Moreover, if $A_{2n} = B_{2l}$, then $y_2 = 2$. In other cases, we have $y_2 =\pm 2.$
\end{lemma}

Finally, we observe the solutions of the system of equations (\ref{eqn:sol_pell_ac}) and (\ref{eqn:sol_pell_ab}) and determine the intersections $x=Q_m=P_l$ of sequences $(Q_m)_m$ and $(P_l)_l$ defined by
\begin{align}
    &P_0=x_2,\ P_1=\frac{1}{2}\left(rx_2+ay_2 \right) ,\ P_{l+2}=rP_{l +1}-P_{l},\label{sequence_P_l}\\
    &Q_0= x_0,\ Q_1=\frac{1}{2}\left(sx_0+az_0 \right),\ Q_{m+2}=sQ_{m +1}-Q_{m}.\label{sequence_Q_m}
\end{align}
From the above, for the equation $x=P_l=Q_m,$ we conclude that only the following
two possibilities exist:\\
\textbf{Type $1$:} If $l\equiv m\equiv 0\pmod{2},$ then $z_0=\pm 2,$ $x_0=2, y_2=\pm2$ and $x_2=2.$\\
\textbf{Type $2$:} If $m\equiv 1\pmod{2},$ then $z_0=\pm t,$ $x_0=r,$ $y_2=\pm 2$ and $x_2=2.$\\

For the rest of this paper, we will carefully examine the following equation: 
\begin{equation}\label{eq-Q_m=P_l}
x=Q_m=P_l,
\end{equation}
while using the fundamental solutions of Types $1$ and $2$. As we mentioned in Remark \ref{c_1-even}, we only need to consider solutions in Type $1$ if $c=c_1^\pm$ since $$\dfrac{1}{2}(cr-st)=\dfrac{1}{2}\left((a+b\pm 2r)r-(r\pm a)(b\pm r) \right)=\pm 2.$$

Let's emphasise which solutions of this equation correspond to the regular extension of our triples to quadruples. For the case $c=c_1^-$ we get $(l,m)=(2,2)$ from fundamental solutions $x_0=2, z_0=2, x_2=2, y_2=-2$ and for the case $c=c_1^+$ we get $(l,m)=(2,2)$ from fundamental solutions $x_0=2, z_0=-2, x_2=2, y_2=2$.
Next, we observe the cases $c=c_2^{\pm}$. Since $d_{-}(a,ka,c_2^{\pm})=c_1^{\pm}$, that extension comes from the solution $(l,m)=(1,1)$ from fundamental solutions $x_0=r, z_0=t, x_2=2, y_2=-2$ (for $c_2^-$) and from fundamental solutions $x_0=r, z_0=-t, x_2=2, y_2=2$ (for $c_2^+$). Next, since $d_{+}(a,ka,c_2^{\pm})=c_3^{\pm}$, that extension comes from the solution $(l,m)=(3,1)$ from fundamental solutions $x_0=r, z_0=-t, x_2=2, y_2=-2$ (for $c_2^-$) and from fundamental solutions $x_0=r, z_0=t, x_2=2, y_2=2$ (for $c_2^+$). For the cases $c=c_3^{\pm}$, we will observe a different equation in the final section.

\subsection{A linear form in three logarithms}\label{subsec-3.2}

Solving recurrences \eqref{sequence_P_l} and \eqref{sequence_Q_m}, we obtain
\begin{align}
\notag P_l&=\frac{1}{2\sqrt{b}}\left((y_2\sqrt{a}+x_2\sqrt{b})\alpha^l-(y_2\sqrt{a}-x_2\sqrt{b})\alpha^{-l} \right),\\
\notag Q_m&=\frac{1}{2\sqrt{c}}\left((z_0\sqrt{a}+x_0\sqrt{c})\beta^m-(z_0\sqrt{a}-x_0\sqrt{c})\beta^{-m} \right),
\end{align} 
where 
\begin{equation}\label{alpha-beta}
\alpha=\dfrac{r +\sqrt{ab}}{2} \quad \mbox{and}\quad \beta=\dfrac{s+\sqrt{ac}}{2}.
\end{equation}
Let us define 
\begin{equation}
\gamma=\frac{\sqrt{c}(y_2\sqrt{a} + x_2\sqrt{b})}{\sqrt{b}(z_0\sqrt{a}+x_0\sqrt{c})}\quad\text{and}\quad \gamma'=\frac{\sqrt{b}(z_0\sqrt{a} + x_0\sqrt{c})}{\sqrt{c}(y_2\sqrt{a}+x_2\sqrt{b})}.
\end{equation}
We define the following linear forms in three logarithms.
\begin{equation}\label{Lambda}
\Lambda = l\log \alpha - m\log \beta + \log \gamma\quad \text{for}\quad c>b,
\end{equation}
and
 \begin{equation}\label{Lambda'}
\Lambda' = m\log \beta - l\log \alpha + \log \gamma'\quad \text{for}\quad c<b.
\end{equation}
Notice that $\Lambda'$ is used only for the case $c_1^{-}$.
\begin{lemma}\label{bound-Lambda} \hfill
\begin{enumerate}[1)]
  \item Assume that $b=ka,\ k\in \{7,8,10,11,12,13\}$ and $c=c_1^-$. If the equation $P_l=Q_m$ has a solution $(l,m)$ of Type $1$ with $l\geq 1$, then
$$
0<\Lambda' < 2.6 \alpha^{-2l}.
$$  
   \item Assume that $b=ka,\ k\in \{7,8,10,11,12,13\}$ and $c \in \{c_1^+,c_2^{\pm}\}$.If the equation $P_l=Q_m$ has a solution $(l,m)$ of Type $1$ with $m\ge 3$, then
$$
0<\Lambda < 2.6 \beta^{-2m}.
$$
If $P_l=Q_m$ has a solution $(l,m)$ of Type $2$ with $m\geq3$, then
$$
0<\Lambda < 1.5 a^2\beta^{-2m}.
$$
\end{enumerate}
\end{lemma}
\begin{proof}
    Proof of $2)$ follows \cite[Lemma 2.6]{glavni} and here we will demonstrate the proof of $1)$.
    Let's define
    \[ E=\frac{z_0\sqrt{a}+x_0\sqrt{c}}{\sqrt{c}}\beta^m, \qquad F=\frac{y_2\sqrt{a}+x_2\sqrt{b}}{\sqrt{b}}\alpha^l\]
    and then our form is $\Lambda'=\log\frac{E}{F}$.
    It is easy to see that $E,F>1$ for $l,m \geq 1$. 
    We can rewrite the equation $P_l=Q_m$ to get
    \[
E + 4\left(\frac{c-a}{c}\right)E^{-1}
= F + 4\left(\frac{b-a}{b}\right)F^{-1}
\mathrel{\overset{b>c}{>}}
F + 4\left(\frac{c-a}{c}\right)F^{-1}
\]
Then $(E - F)\left(EF - 4\frac{c-a}{c}\right) > 0$. Since $EF > 4\left(\frac{c-a}{c}\right)$ we get $E>F$. It follows that $\Lambda'>0$ and \[\Lambda'=\log(1+\frac{E-F}{F})<\frac{E-F}{F}<4\left(\frac{c-a}{c}\right)F^{-2}<4F^{-2}.\]
Now, since we observe only the solutions in Type $1$ for $c=c_1^-$, we have
$$\Lambda'<4\frac{b}{(\pm2\sqrt{a}+2\sqrt{b})^2}\alpha^{-2l}<\frac{k}{(\sqrt{k}-1)^2}\alpha^{-2l}<2.6\alpha^{-2l}. $$

\end{proof}
\begin{lemma} \hfill
\begin{enumerate}[1)]
    \item Assume that $b=ka,\ k\in \{7,8,10,11,12,13\}$ and $c=c_1^-$. If $P_l=Q_m$ has a solution $(l,m)$ with $l \geq 1$, then $l \leq m$.
    \item Assume that $b=ka,\ k\in \{7,8,10,11,12,13\}$ and $c \in \{c_1^+,c_2^{\pm}\}$. If $P_l = Q_m$ has a solution $(l,m)$ with $m \geq 3$, then $m \leq l$.
\end{enumerate}
\end{lemma}
\begin{proof}
    Proof of $2)$ follows \cite[Lemma 8]{ahpt} and here we will demonstrate the proof of $1)$. Since we proved in the previous Lemma that $\Lambda'>0$, we have
    \[\frac{m}{l}>\frac{\log\alpha}{\log\beta}-\frac{\log\gamma'}{l\log\beta}.\] To prove our statement, we need to show that the right-hand side is greater than $1-\frac{1}{l}$, which is equivalent to proving that
    \[\left(\frac{\alpha}{\beta}\right)^l>\frac{\gamma'}{\beta}.\]
    Since $ac>a^2>10864^2$, we have
    \[\frac{\alpha}{\beta}=\frac{r+\sqrt{ab}}{s+\sqrt{ac}}=\frac{\sqrt{1+\frac{4}{ab}}+1}{\sqrt{1+\frac{4}{ac}}+1}\sqrt{\frac{b}{c}}>\frac{2}{2.1}\sqrt{\frac{b}{c}}.\]
    Also,
   \begin{align*}
\frac{b}{c} 
&= \frac{ka}{c_1^-} 
= \frac{ka}{(k+1)a - 2\sqrt{ka^2 + 4}} 
= \frac{ka^2\!\left(k+1 + 2\sqrt{k + \frac{4}{a^2}}\right)}{(k-1)^2 - \frac{16}{a^2}} \\
&> \frac{k}{(k-1)^2}\!\left(k+1 + 2\sqrt{k}\right)
= \frac{k}{(\sqrt{k}-1)^2}.
\end{align*}
Since $l\geq1$ we get \[\left(\frac{\alpha}{\beta}\right)^l>1.31.\]
On the other hand, since for $c=c_1^-$ we only observe fundamental solutions of Type $1$ we get
\begin{align*}
    \frac{\gamma'}{\beta}&=\frac{2\sqrt{b}(\pm2\sqrt{a}+2\sqrt{c})}{(s+\sqrt{ac})\sqrt{c}(2\sqrt{a}+2\sqrt{b})}<\frac{\sqrt{b}(\sqrt{a}+\sqrt{c})}{c\sqrt{a}(-\sqrt{a}+\sqrt{b})} \\
    &=\frac{\sqrt{k}(\sqrt{a}+\sqrt{c})}{c(\sqrt{k}-1)\sqrt{a}}\leq\frac{2\sqrt{k}}{(\sqrt{k}-1)a}<1.
    \end{align*}
    Finally, we have
    \[\left(\frac{\alpha}{\beta}\right)^l>1.31>1>\frac{\gamma'}{\beta}.\]

\end{proof}

For any nonzero algebraic number $\alpha$ of degree $d$ over $\QQ$ whose minimal polynomial over $\ZZ$ is $a_0 \prod_{j=1}^{d}(X-\alpha^{(j)}),$ we denote by 
$$ 
 h(\alpha)= \frac{1}{d}\left(\log |a_0|+\sum_{j=1}^{d}\log\max\left(1,\left|\alpha^{(j)} \right| \right) \right)
$$ 
its absolute logarithmic height. We recall the following result due to Matveev \cite{Matveev:2000}.
\begin{lemma}\label{Matveev}
Denote by $\alpha_1,\ldots,\alpha_j$ algebraic numbers, not $0$ or $1,$ by $\log\alpha_1,\ldots,$ $\log\alpha_j$ determinations of their logarithms, by $D$ the degree over $\QQ$ of the number field $\KK=\QQ(\alpha_1,\ldots,\alpha_j),$ and by $b_1,\ldots,b_j$ integers. Define $B=\max\{|b_1|,\ldots,|b_j|\}$ and 
$$A_i=\max\{Dh(\alpha_i),|\log\alpha_i|,0.16\}\;\; (1\leq i\leq j),$$
where $h(\alpha)$ denotes the absolute logarithmic Weil height of $\alpha.$ 
Assume that the number 
$$
\Lambda= b_1\log\alpha_1+\cdots+b_n\log\alpha_j 
$$ 
does not vanish. Then 
$$ 
|\Lambda|\geq \exp\{-C(j,\chi)D^2A_1\cdots A_j\log(eD)\log(eB)\},
$$ 
where $\chi=1$ if $\KK \subset\RR$ and $\chi=2$ otherwise and 
$$ 
C(j,\chi)=\min\left\{\frac{1}{\chi}\left( \frac{1}{2}ej\right)^{\chi}30^{j+3}j^{3.5},2^{6j+20}\right\}.
$$
\end{lemma}
By applying this result we obtain upper bound for $l$ and $m$ in terms of $a$.
\begin{proposition}\label{proposition-1}
Assume that $c\in \{c_1^+, c_2^\pm\}.$ If $P_l=Q_m,$ with $m>1$ then
\begin{align*}
\frac{l}{\log(el)}&< 3.34\cdot 10^{13}\cdot\log^2(8.09c^2),\quad \text{with solutions of Type $1$},\; \\
\frac{l}{\log(el)}&< 6.63\cdot 10^{13}\cdot\log^2(8.09c^2), \quad \text{with solutions of Type $2$}.
\end{align*}
If $Q_m=P_l,\ l\geq1$ with $c=c_1^-,$ then we get 
\begin{align*}
\frac{m}{\log(em)}< 13.36\cdot 10^{13}\cdot\log^2(21.3a),\quad \text{for}\; k=7,8,10,11,12,13. 
\end{align*}
% and 
% \begin{align*}
% \frac{l}{\log(el)}< 9\cdot 10^{13}\cdot\log^2(27a),\quad \text{if}\; k=6.
%\end{align*} 
\end{proposition}
\begin{proof}
    We will demonstrate the proof for the case $c=c_1^-$ since the proof for the other cases closely follows \cite[Proposition 2.9]{glavni}. We apply Lemma \ref{proposition-1}
 with $j=3$ and $\chi=1$ to the linear form \ref{Lambda'} and take \[D=4,\ b_1=m,\ b_2=-l,\ b_3=1,\ \alpha_1=\beta,\ \alpha_2=\alpha,\ \alpha_3=\gamma'.\]
 Since $l\leq m$, we can take $B=m$. Also, we have
 \[h(\alpha_1)=\frac{1}{2}\log\beta,\ h(\alpha_2)=\frac{1}{2}\log\alpha.\]
 Since $\gamma'=\gamma^{-1}$, then $h(\gamma)=h(\gamma')$ and from \cite[Proposition 2.9]{glavni} we have
 \[h(\gamma)<\frac{1}{4}\log\left[\frac{2^4r^4c^4(1+\sqrt{k})^4}{(c-a)^2}\right].\]
 We have \begin{align*}(c_1^--a)^2&=a^2(2\sqrt{k+\frac{4}{a^2}}-k)^2>a^2(2\sqrt{k}-k)^2,\\
 r&<a\sqrt{1.1k},\\ c=c_1^-&=a(k+1-2\sqrt{k+\frac{4}{a^2}})<a(\sqrt{k}-1)^2,\end{align*} and now it follows that
 \[h(\gamma')=h(\gamma)<\frac{1}{4}\log(93309063a^6)<\frac{3}{2}\log(21.3a).\]
 By applying Lemma \ref{proposition-1} with
 \[A_1=2\log\beta,\ A_2=2\log\alpha,\ A_3=6\log(21.3a)\]
 we get
 \begin{align}\label{Matveev_0}
\log|\Lambda'|>-1.3901&\cdot 10^{11}\cdot16\cdot2 \cdot\log\beta\\&\cdot2\log\alpha\cdot6\log(21.3a)\cdot\log(4e)\cdot\log(em).\nonumber
\end{align}
From $1)$ of Lemma \ref{bound-Lambda} and using $l\geq1$, it is easy to conclude that $$m\log\beta<l\log\alpha+2.6\alpha^{-2l}-\log\gamma'<2l\log\alpha.$$
Also,
$$\log|\Lambda'|<-1.9069l\log\alpha$$
and
$$\log\alpha=\log\left[a\left(\sqrt{k+\frac{4}{a^2}}+\sqrt{k}\right)\right]<\log(7.212a)<\log(21.3a).$$
Now, by combining everything, we get
\[\frac{m}{\log(em)}<1.336\cdot10^{14}\cdot\log^2(21.3a),\ \text{for}\ k=7,8,10,11,12,13. \qedhere\]
\end{proof}

%%%%%%%%%%%%%%%%%%%%%%%%%%%%%%%%%%%%%%%%%%%%%%%%%%%%%%%%%%%%%%%%%%%%%%

\section{Lower bounds for $m$ and $l$}\label{section_5}

In this section, we examine the equation \[x=Q_m=P_l.\] Firstly, we state a useful result from \cite{glavni} and then apply it to obtain lower bounds for the indices $m$ and $l$ in terms of $a$.
\begin{lemma}\cite[Lemma 3.1]{glavni}\label{lem_6} If $a$ is odd, then 
\begin{align}
\label{cong_1} Q_{2m}&\equiv  x_0+\dfrac{1}{2} a(cx_0m^2+sz_0m)\pmod{a^2},\\
P_{2l}&\equiv  x_2+\dfrac{1}{2} a(bx_2l^2+ry_2l)\pmod{a^2} .
\end{align}
If $a$ is even, then 
\begin{align}
 Q_{2m}&\equiv  x_0+\dfrac{1}{2} a(cx_0m^2+sz_0m)\pmod{\dfrac{1}{2} a^2},\\
P_{2l}&\equiv  x_2+\dfrac{1}{2} a(bx_2l^2+ry_2l)\pmod{\dfrac{1}{2} a^2} .
\end{align}
\end{lemma}
\begin{lemma}\label{lower-Type-a} \hfill
\begin{enumerate}[1)]
  \item Assume that $b=ka,\ k\in \{7,8,10,11,12,13\}$ and $c=c_1^-$. If the equation $P_l=Q_m$ has a solution $(l,m)$ (of Type $1$), then we have
$$
m \geq \frac{1}{4}\left(-2+\sqrt{4+\sqrt{a}}
\right).$$  
   \item Assume that $b=ka,\ k\in \{7,8,10,11,12,13\}$ and $c \in \{c_1^+,c_2^{\pm}\}$. If the equation $P_l=Q_m$ has a solution $(l,m)$ (of Type $1$), then we have
$$
l \geq \frac{1}{12}\left(-2+\sqrt{4+3\sqrt{a}}
\right).$$  
\end{enumerate}
\end{lemma}
\begin{proof}
    Proof follows \cite[Lemma 3.2]{glavni} and here we will demonstrate the proof of $1)$. From the recurrent sequence (\ref{sequence_s_nu}), we have $s \equiv 2,r\ (\bmod\ a)$. Also, $b=ka \equiv 0\ (\bmod\ a)$ and $c=c_{1}^{-}\equiv -2 r\ (\bmod\ a)$. Using the previous Lemma, with solutions in Type $1$, we get
    $$  
- 4 r m^2\pm 4m\equiv \pm2rl\pmod{a},\quad \mbox{if}\quad s\equiv 2\pmod{a}
$$
and 
$$  
\ - 4m^2\pm 2m\mp2l\equiv 0\pmod{\dfrac{a}{\gcd (a,r)}},\quad \mbox{if}\quad s\equiv r\pmod{a}.
$$
In the first case, we multiply the congruence by $r$ and since $r^2 \equiv 4\ (\bmod\ a)$ we have $r \equiv \pm2\ (\bmod\ a')$ for some $a'$ that divides $a$ and $a' \geq \sqrt{a}$. So we get
$$-16m^2\pm8m\mp8l \equiv 0\ (\bmod\ a').$$
Using $l\leq m$ it follows that
$$16m^2+16m \geq |-16m^2\pm8m\mp8l|\geq a'\geq \sqrt{a},$$
which implies
\begin{align}\label{1}
    m \geq \frac{1}{4}\left(-2 +\sqrt{4+\sqrt{a}}  \right).\end{align}
In the second case, since $\gcd(a,r)\leq2$ and $l \leq m$, we have
$$4m^2+4m \geq |- 4m^2\pm 2m\mp2l|\geq\frac{a}{\gcd(a,r)}\geq\frac{a}{2},$$
which implies
\begin{align}\label{2}
    m \geq \frac{1}{4}\left(-2+\sqrt{4+2a}\right).
\end{align}
Combining (\ref{1}) and (\ref{2}) we obtain the desired inequality.
\end{proof}

In \cite{glavni} it is demonstrated that $s \equiv \pm2,\pm a\ (\bmod\ r)$ and that the case $s \equiv \pm a\ (\bmod\ r)$ leads to a contradiction if $c=c_2^{\pm}$, which also applies here. Therefore, for solutions of Type $2$, we obtain another lower bound on $l$ and $m$. With only slight changes in calculations, the next result follows similarly as \cite[Lemma 3.4]{glavni}.
\begin{lemma}\label{lower-Type-b}
Assume that $c=c_2^\pm.$ If the equation $P_l=Q_m$ has a solution $(l,m)$ of Type $2$, then we have
$$
m>  \begin{cases}
(14-5\sqrt{7.1})a/4,&\ k=7,\\
(6\sqrt{8}-16)a/4,&\ k=8,\\
(20-6\sqrt{10.1})a/4,&\ k=10, \\
(7\sqrt{11}-22)a/4,&\ k=11,\\
(7\sqrt{12}-24)a/4,&\ k=12,\\
(26-7\sqrt{13.1})a/4,&\ k=13.
\end{cases}
$$
\end{lemma}

\section{Proof of Theorem \ref{main result}}\label{section_6}
In this section, we complete the proof of Theorem \ref{main result} in two subsections according to the values of c.

\subsection{Proof of Theorem \ref{main result} for $c=c_1^\pm, c_2^\pm$} \label{subsec5.1} \hfill

Combining (\ref{a_p_1})-(\ref{a_p_6}) with Proposition~\ref{proposition-1}, Lemmas~\ref{lower-Type-a} and \ref{lower-Type-b} we obtain the following result.

\begin{lemma} \hfill
    \begin{enumerate}[1)]
\item For the $D(4)$-triples $\{a,ka,c_1^-\}$ with $a=a_p^{(k)}$ $(p\ge 1)$ defined in \eqref{a_p_1}-\eqref{a_p_6}, if the equation $P_l=Q_m$ has a solution $(l,m)$ with $l\ge 1,$ then $p\leq E_k$ and $m \leq 2.48\cdot10^{20}$ where $E_k \in \{70,110,53,65,147,81\}$ for $k \in \{7,8,10,11,12,13\}$ respectively.
\item For the $D(4)$-triples $\{a,ka,c_1^+\}$ with $a=a_p^{(k)}$ $(p\ge 1)$ defined in \eqref{a_p_1}-\eqref{a_p_6}, if the equation $P_l=Q_m$ has a solution $(l,m)$ with $m\ge 3,$ then $p\leq E_k'$ and $l\le 2.56\cdot10^{20}$ where $E_k' \in \{71,111,54,65,149,82\}$ for $k \in \{7,8,10,11,12,13\}$ respectively.
\item For the $D(4)$-triples $\{a,ka,c_2^{\pm}\}$ with $a=a_p^{(k)}$ $(p\ge 1)$ defined in \eqref{a_p_1}-\eqref{a_p_6}, if the equation $P_l=Q_m$ has a solution $(l,m)$ in Type $1$  with $m\ge 3,$ then $p\le E_k''$ and $l\le 2.6\cdot10^{21}$ where $E_k'' \in \{74,116,56,68,$
$156,86\}$ for $k \in \{7,8,10,11,12,13\}$ respectively. If the equation $P_l=Q_m$ has a solution $(l,m)$ in Type $2$  with $m\ge 3$ then $p\le E_k'''$ and $l\le 3.13 \cdot 10^{20} $ where $E_k''' \in \{18,28,13,16,38,21\}$ for $k \in \{7,8,10,11,12,13\}$ respectively.
\end{enumerate}
\end{lemma}

For the remaining cases, we will use the following lemma which is a slight modification of the original version of Baker-Davenport reduction method (see \cite[Lemma~5a]{Dujella-Pethoe:1998}). 

\begin{lemma}\label{Dujella-Pethoe}
Assume that $M$ is a positive integer. Let $p/q$ be a convergent of the continued fraction expansion of $\kappa$ such that $q>6M$ and let 
$$ \eta=\parallel\mu q\parallel-M\cdot\parallel\kappa q\parallel,$$ 
where $ \parallel\cdot\parallel $ denotes the distance from the nearest integer. If $\eta>0,$, then there is no solution to the inequality 
$$0<l\kappa-m+\mu<AB^{-l}$$ 
in integers $l$ and $m$ with 
$$\frac{\log(Aq/\eta)}{\log(B)}\leq l\leq M.$$
\end{lemma}
In order to apply Lemma \ref{Dujella-Pethoe} we define parameters depending on $c.$

\textbf{The case $c=c_1^-$}. Dividing $0<\Lambda'<2.6\alpha^{-2l}$  by $\log\alpha$ and using the fact that $\alpha^{-2l}<\beta^{-m}$ we get
\begin{equation}
0<m\kappa-l+\mu<AB^{-m},
\end{equation} 
where 
$$
\kappa:=\frac{\log\beta}{\log\alpha},\quad\mu:=\frac{\log\gamma'}{\log\alpha},\quad A:=\frac{2.6}{\log\alpha},\quad B:=\beta.
$$

\textbf{The case $c\in\{c_1^+, c_2^{\pm}\}$}. Dividing $0<\Lambda<2.6\beta^{-2m}$ and $0<\Lambda<1.5a^2\beta^{-2m}$ by $\log\beta$ and using the fact that we have $\beta^{-2m}<\alpha^{-l}$ leads to an inequality of the form \begin{equation}\label{last}
0<l\kappa-m+\mu<AB^{-l},
\end{equation} 
where we consider solutions of Type $1$
$$
\kappa:=\frac{\log\alpha}{\log\beta},\quad\mu:=\frac{\log\gamma}{\log\beta},\quad A:=\frac{2.6}{\log\beta},\quad B:=\alpha,
$$
and for solutions of Type $2$
$$
\kappa:=\frac{\log\alpha}{\log\beta},\quad\mu:=\frac{\log\gamma}{\log\beta},\quad A:=\frac{1.5a^2}{\log\beta},\quad B:=\alpha.
$$

Let's first observe the case $c=c_1^-$. After at most four steps of reduction, we find that $P_l=Q_m$ implies $1\leq l\leq m \leq11$ in all cases. Combining this with 
Lemma \ref{lower-Type-a}, we get $a \leq 4460544$. Then we explicitly verify these remaining cases and find that the only solution for the equation $P_l=Q_m$ is $(l,m)=(2,2)$ (from fundamental solutions $x_0=2,z_0=2,x_2=2,y_2=-2$), which corresponds to the regular extension of a triple to a quadruple. For $l=0$ we get $x=Q_0=P_0=2$, which gives $d=0$. 

Now, for the cases $c \in \{c_1^+,c_2^{\pm} \}$, we find that after at most four steps of reduction, $P_l=Q_m$ implies $3 \leq m\leq l\leq 8$ in all cases. Combining this with Lemma \ref{lower-Type-a} for solutions in Type $1$, we get $a \leq 10 240 000$. We then explicitly verify these remaining cases and find that the equation $P_l=Q_m$ has no solutions in this range. For solutions in Type $2$, we combine $3 \leq m \leq l \leq 8$ with Lemma \ref{lower-Type-b}, and in all cases, we obtain $a \leq 128$, which contradicts the fact that $b=ka >10^5,\ k=7,8,10,11,12,13$.
Since the relation $m \leq l$ only holds when $m \geq 3$, the final cases $m \in \{0,1,2\}$ are observed in the same way as in \cite{glavni}. We conclude that the only possible intersection $P_l=Q_m$ (besides the trivial $(l,m)=(0,0)$) is the one that corresponds to the regular extension of a triple to a quadruple.

\subsection{Proof of Theorem \ref{main result} with $c=c_3^\pm$}\label{subsec5.2}
In this case, we examine the equation $z=v_m=w_n$ using Lemma \ref{lem:inital_terms}. By \cite[Lemma 5]{Filipin-2009}, we know that if this equation has a solution $(m,n)$, then $n-1\leq m\leq 2n+1$. We now examine the solutions for $2<n<m<2n$. The next result follows as in \cite[Lemma 4.5]{glavni}.

\begin{lemma}\label{lower-n-m}
\begin{enumerate}[i)]
    \item If the equation $z=v_{2m}=w_{2n}$ has a solution $(m,n)$ with $n > 2,$ then $m>0.495b^{-0.5}c^{0.5}.$
    \item If the equation $z=v_{2m+1}=w_{2n+1}$ has a solution $(m,n)$ with $n > 2,$ then $m^2>0.0625b^{-1}c^{0.5}.$
\end{enumerate}
\end{lemma}

Filipin proved in \cite{Filipin-2008} that $z=v_m=w_n,$ for $n>2,$ implies
\begin{align}
\dfrac{m}{\log (m+1)}< 6.543\cdot 10^{15} \log^2 c.
\end{align}
Combining this with Lemma \ref{lower-n-m}, in the case of even indices, we get
\begin{equation}\label{mn-p-even}
\frac{2\cdot0.495b^{-0.5}c^{0.5}}{\log(2\cdot0.495b^{-0.5}c^{0.5}+1)}<6.543\cdot10^{15}\log^2c,
\end{equation}
and in the case of odd indices, we get the inequality
\begin{equation}\label{mn-p-odd}
\frac{2\cdot0.0625^{0.5}b^{-0.5}c^{0.25}+1}{\log(2\cdot0.0625^{0.5}b^{-0.5}c^{0.25}+2)}<6.543\cdot10^{15}\log^2c.
\end{equation}
The solutions obtained from these inequalities are summarized in the following lemma.

\begin{lemma}\label{bounds}
For the $D(4)$-triples $\{a,ka,c_3^\pm\}$ with $a=a_p^{(k)}$ $(p\ge 1)$ defined in \eqref{a_p_1}-\eqref{a_p_6}, if $z=v_{2m}=w_{2n}$ has a solution $(m,n)$, then $p\le E_k$  and $m\le 5.2\cdot10^{21}$ where $E_k \in \{9,14,6,8,19,10\}$ for $k\in \{7,8,10,11,12,13\}$ respectively. However, if $z=v_{2m+1}=w_{2n+1}$ has a solution $(m,n)$, then $p\le F_k$ and $m\le 4.3\cdot10^{22},$ where $F_k \in \{25,40,19,23,53,29\}$ for $k\in \{7,8,10,11,12,13\}$ respectively.
\end{lemma}

Now, for the remaining small values of $p$, by using Lemma \ref{Dujella-Pethoe} and Lemma \ref{bounds} in the same way as in \cite{glavni} we get that $z=v_m=w_n$ implies $n \leq m \leq 2$. In these small ranges, we verify that all solutions of $z=v_m=w_n$ give the extension of a $D(4)$-triple $\{a,b,c\}$ to a regular quadruple. This completes the proof of Theorem \ref{main result}.

\section*{Acknowledgements}
The authors are supported by the Croatian Science Foundation, grant HRZZ IP-2022-10-5008.

%\section*{Conflict of interest statement}
%Not Applicable.

%\section*{Data availability statement}
%Not Applicable.

%%%%%%%%%%%%%%%%%%%%%%%%%%%%%%%%%%%%%%%%%%%%%%%%%%%%%%%%%%%%%%%%%%%%%%

%%%%%%%%%%%%%%%%%%%%%%%%%%%%%%%%%%%%%%%%%%%%%%%%
\vspace{20pt}
University of Split, Faculty of Science, Ru\dj{}era Bo\v{s}kovi\'{c}a 33,
21000 Split, Croatia\\
Email: marbli@pmfst.hr \\[6pt]
University of Split, Faculty of Science, Ru\dj{}era Bo\v{s}kovi\'{c}a 33,
21000 Split, Croatia\\
Email: pradic@pmfst.hr \\[6pt]

\end{document}